\documentclass[preprint,11pt]{elsarticle}
\usepackage{amssymb}
\usepackage{amsthm}
\usepackage{graphics}
\usepackage{epsfig}
\usepackage{amssymb}
\usepackage{graphicx}
\usepackage{subfigure}
\usepackage{color}
\usepackage{tikz}
\usepackage{tikz,pgfplots}
\pgfplotsset{compat=1.12,axis lines=center}
\usepgfplotslibrary{fillbetween}
\usetikzlibrary{patterns}
\usepackage{hyperref}
\usepackage{a4wide}
\usepackage{amsmath}
\usepackage{amsfonts}
\usepackage{enumerate}
\usepackage{multicol}

\makeatletter
\def\ps@pprintTitle{%
	\let\@oddhead\@empty
	\let\@evenhead\@empty
	\let\@oddfoot\@empty
	\let\@evenfoot\@oddfoot
}
\makeatother

\hypersetup{
	colorlinks=true,
	linkcolor=blue,
	filecolor=dviolet,      
	urlcolor=blue,
} 

\newtheorem*{theoremaux}{Theorem \theoremauxnum}
\gdef\theoremauxnum{1}

	\newtheorem{lemma}{\bf Lemma}[section]
	
	\newtheorem{theorem}{\bf Theorem}[section]
	\newtheorem{proposition}[lemma]{\bf Proposition}
	\newtheorem{corollary}[lemma]{\bf Corollary}
	\newtheorem{definition}{\bf Definition}[section]

\journal{~}

\begin{document}

	\begin{frontmatter}
		
		
		
		\title{On the strong domination number of proper enhanced power graphs of finite groups}

		
		
		\author{Sudip Bera}
		\ead{sudip\_bera@daiict.ac.in} 
		\address{Faculty of Mathematics,\\ 
		DA-IICT, Gandhinagr \\ India} 
		%
		
		
		\begin{abstract}
			\noindent 
			 The enhanced power graph of a group $G$ is a graph with vertex set $G$, where two distinct vertices $x$ and $y$ are adjacent if and only if there exists an element $w$ in $G$ such that both $x$ and $y$ are powers of $w$. To obtain the proper enhanced power graph, we consider the induced subgraph on the set $G \setminus D$, where $D$ represents the set of dominating vertices in the enhanced power graph. In this paper, we aim to determine the strong domination number of the proper enhanced power graphs of finite nilpotent groups.
		\end{abstract}
		
		
	\end{frontmatter}
\section{Introduction}
The examination of graphs associated with various algebraic structures has gained significance in the past twenty years. The study of these graphs offers multiple advantages. Firstly, it allows us to categorize the resulting graphs. Secondly, it enables us to identify algebraic structures that possess isomorphic graphs. Lastly, it helps us understand the interdependence between algebraic structures and their corresponding graphs. Additionally, these graphs have numerous valuable applications, as evidenced by sources such as \cite{surveypwrgraphkac1, cayleygraphsckry}. Moreover, they are closely linked to automata theory \cite{automatatheory}. There exist various types of graphs, including but not limited to commuting graphs of groups \cite{ijraeljofmathematics, braurflower, onboundingdiamcommuting}, power graphs of semigroups \cite{undpwrgraphofsemgmainsgc1, directedgrphcompropofsemgrpkq3}, groups \cite{combinatorialpropertyandpowergraphsofgroupskq1}, intersection power graphs of groups \cite{intersectionpwegraphb3}, enhanced power graphs of groups \cite{firstenhcedpwrstrctreaacns1,enhancedpwrgrapbb3,Jitendra-Ma-acta-enhaced}, and comaximal subgroup graphs \cite{das-saha-saba}. These graphs have been established to explore the properties of algebraic structures using graph theory.

The commuting graph, which is one of the most significant and extensively studied graphs, is associated with a group $G$. This graph has been thoroughly examined in \cite{ijraeljofmathematics, braurflower, onboundingdiamcommuting}. In order to provide a formal definition of the commuting graph, we refer to the following:

\begin{definition} [\cite{ijraeljofmathematics, braurflower, onboundingdiamcommuting}]
	Consider a group $G$. The commuting graph of $G$, denoted as $\mathcal{C}(G)$, is a simple graph where the vertex set consists of non-central elements of $G$. Two distinct vertices $u$ and $v$ are connected in this graph if and only if $u$ and $v$ commute, meaning $uv=vu$.
\end{definition}
Furthermore, Segev and Seitz \cite{Segev-Seitz-Pacific} utilized combinatorial parameters of specific commuting graphs to establish conjectures in the theory of division algebras. Additionally, a variant of commuting graphs on groups has played a crucial role in the classification of finite simple groups, as mentioned in \cite{finite-smple-group-book}.

Around the year 2000, Kelarev and Quin \cite{combinatorialpropertyandpowergraphsofgroupskq1} introduced the concept of a power graph in the context of semigroup theory.
\begin{definition}[\cite{surveypwrgraphkac1,undpwrgraphofsemgmainsgc1, combinatorialpropertyandpowergraphsofgroupskq1}]\label{defn: powr graph}
	Given a group $G,$ the \emph{power graph} $\mathcal{P}(G)$ of $G$ is a simple graph with vertex set $G$ and two vertices $u$ and $v$ are connected by an edge if and only if one of them is the power of another.  
\end{definition}
A new graph known as the \emph{enhanced power graph} was introduced by Aalipour et al. \cite{firstenhcedpwrstrctreaacns1} to assess the proximity between the power graph and the commuting graph of a group G. Following this, the closeness of the power graph to the commuting graph can be measured.
\begin{definition}[\cite{firstenhcedpwrstrctreaacns1}]\label{defn: enhcdpowr graph}
	Let $G$ be a group. The \emph{enhanced power graph} of $G,$ denoted by $\mathcal{G}_E(G),$ is the graph with vertex set $G,$ in which two vertices $u$ and $v$ are joined if and only if there exists an element $w \in G$ such that both $u \in \langle w \rangle $ and $v \in  \langle w \rangle.$   
\end{definition}
Numerous studies have recently been conducted to examine various characteristics of the enhanced power graph of finite groups. In \cite{firstenhcedpwrstrctreaacns1}, the authors provided a characterization of finite groups in which any arbitrary pair of the power, commuting, and enhanced graphs are identical. Additionally, \cite{Zahirovienhnacedpwrgraph} presented a proof that finite groups with isomorphic enhanced power graphs also have isomorphic directed power graphs. The work by Bera et al. in \cite{enhancedpwrgrapbb3} focused on investigating the completeness, dominatability, and other intriguing properties of the enhanced power graph. Ma and She, in \cite{ma-she}, derived the metric dimension, while Hamzeh et al., in \cite{Hamzeh-ashrafi}, determined the automorphism groups of enhanced power graphs of finite groups. Recently, Bera et al., in \cite{bera-dey-JGT}, explored the connectivity, dominatibility, diameter, and Laplacian spectral radius of the proper enhanced power graphs (defined later) of groups. The main focus of our paper is the strong domination number of the proper enhanced power graphs of groups.
\subsection{Basic definitions and notations}\label{subsection:definitions}
For the convenience of the reader and also for later use, we recall some basic definitions and notations about graphs. Let $\Gamma=(V, E)$ be a simple undirected graph with vertex set $V$ and edge set $E.$ Two elements $u$ and $v$ are said to be adjacent if there is an edge between them. For a vertex $u,$ we denote by $\text{nbd}(u)$ the set of vertices which are adjacent to $u.$ A \emph{path} of length $k$ between two vertices $v_0$ and $v_k$ is an alternating sequence of vertices and edges $v_0, e_0, v_1, e_1, v_2, \cdots , v_{k-1}, e_{k-1}, v_k$, where the $v_i'$s are distinct (except possibly the first and last vertices) and $e_i'$s are the edges $(v_i, v_{i+1}).$  A graph $\Gamma$ is said to be \emph{connected} if for any pair of vertices $u$ and $v,$ there exists a path between $u$ and $v.$ The distance between two vertices $u$ and $v$ in a connected graph $\Gamma$ is the length of the shortest path between them and it is denoted by $d(u, v).$ Clearly, if $u$ and $v$ are adjacent, then $d(u, v)=1.$ A graph $\Gamma$ is said to be \emph{complete} if any two distinct vertices are adjacent. A vertex of a graph $\Gamma$ is called a \emph{dominating vertex} if it is adjacent to every other vertex. For a graph $\Gamma,$ let $\text{\emph{Dom}}(\Gamma)$ denote the set of all dominating vertices in $\Gamma.$    
From the definition of the enhanced power graph, it is clear that the identity of the group is always a dominating vertex. The enhanced power graph is called \emph{dominatable} if it has a dominating vertex other than identity. For more on graph theory we refer \cite{graphthrybondymurti, algbraphgodsil, graphthrywest}. Throughout this paper we consider $G$ as a finite group.  $|G|$ denotes the cardinality of the set $G.$ For a prime $p,$ a group $G$ is said to be a $p$-group if $|G|=p^{r}, r\in \mathbb{N}.$ If $|G|= p_{\ell}^{r}$ for some prime $p_{\ell}$, then we say that $G$ is a $p_{\ell}$-group. For a subgroup $H$ of $G$, we call $H$ a $p$-order subgroup if $|H|=p.$ For any element $g \in G, \text{o}(g)$ denotes the order of the element $g.$ Let $m$ and $n$ be any two positive integers, then the greatest common divisor of $m$ and $n$ is denoted by $\text{gcd}(m, n).$ The Euler's phi function $\phi(n)$ is the number of integers $k$ in the range $1 \leq k \leq n$ for which the  $\text{gcd}(n, k)$ is equal to $1.$ The set $\{1, 2, \cdots, n\}$ is denoted by $[n].$ If $n$ is a positive integer, then $\pi(n)$ denotes the set of prime divisors of $n.$ For a group $G,$ set $\pi(G) = \pi(|G|).$ Fix a set of prime numbers $\pi$. An element $x\in G$ is called a $\pi$-element if every prime divisor of $\text{o}(x)$ is a member of $\pi.$ 

The structure of the paper is outlined as follows. Our main results are presented in Section \ref{Se: main results}, while some previously known results are mentioned in Section \ref{Sec:Preliminaries}. Section \ref{Se: proofs of main results} focuses on the study of the proofs for our main theorems.

\section{Main results}\label{Se: main results}
In this section, we state and motivate the main results of this paper. For that purpose, we first recall the following definitions.
\begin{definition}
	Given an undirected graph $\Gamma=(V, E)$ a set $D \subseteq V$ is said to be a \emph{dominating set} if every vertex of $V\setminus D$ is adjacent to some vertex of $D.$ The domination number of $\Gamma$ is defined as: $\gamma(\Gamma)=\text{min}\{|D|: D \text{ is a dominating set of } \Gamma\}.$ 
\end{definition}
\begin{definition}	
	A total dominating set (or strongly-dominating set) is a set of vertices such that all vertices in the graph, including the vertices in the dominating set themselves, have a neighbor in the dominating set. That is: for every vertex 
	$u\in V, \exists$ a vertex $v\in D$ such that $u\sim v$ in $\Gamma.$ The strong domination number of $\Gamma$ is defined as: $\gamma^{\text{strong}}(\Gamma)=\text{min}\{|D|: D\text{ is a strongly dominating set of } \Gamma\}.$
\end{definition}
The study of the domination problem began in the 1950s and gained significant momentum in the mid-1970s. Dominating sets in graph theory have various applications in communication networks. The minimum dominating set of sites plays a crucial role in network domination, ensuring the entire network is covered at the lowest possible cost. Graph domination has practical implications in several fields, including facility location problems, where minimizing travel distance to the nearest facility is essential. Additionally, graph domination is utilized in safeguarding medical information networks against cyber threats see \cite{D-Angel}. 
In a recent publication by Bera and Dey \cite{bera-dey-JGT}, they successfully established the domination number for proper enhanced power graphs of nilpotent groups that do not possess a generalized quaternion Sylow subgroup. In this study, we focus on computing the strong domination number for these graphs. Additionally, we investigate the domination number for proper enhanced power graphs of nilpotent groups that do have a generalized quaternion Sylow subgroup. To address this inquiry, we initiate our study by introducing the concept of the proper enhanced power graph. 
\begin{definition}\label{defn: proper enhacd pwr graph}	Given a group $G,$ the \emph{proper enhanced power graph} of $G,$ denoted by $\mathcal{G}^{**}_E(G),$ is the graph obtained by deleting all the dominating vertices from the enhanced power graph $\mathcal{G}_E(G).$ Moreover, by $\mathcal{G}^{*}_E(G)$ we denote the graph obtained by deleting only the identity element of $G.$  
\end{definition}
We are now prepared to present the main theorems of this article. Initially, we will describe all finite nilpotent groups $G$ for which the proper enhanced power graph possesses a total dominating set. To achieve this, we have the subsequent results: 
\begin{theorem}\label{Thm: existence of srong dom set in enhd pwr grph}
	If $G$ is a finite nilpotent group, then the graph $\mathcal{G}_E^{**}(G)$ will not possess any total dominating set if and only if the following conditions hold true:
	\begin{enumerate}
		\item $G$ is a $2$-group and 
		\item $G$ has an element of order $2$ which is not contained in any other proper subgroup of order $2^k, k\geq 2.$	
	\end{enumerate}
\end{theorem}
Now we will examine all finite nilpotent groups $G$ for which the proper enhanced power graph $\mathcal{G}_E^{**}(G)$ possesses a total dominating set. In this analysis, we will determine the strong domination number of these graphs. To begin, we must first establish the characteristics of a nilpotent group.

A finite group $G$ is classified as nilpotent if and only if it can be expressed as the direct product of its Sylow subgroups. Each Sylow subgroup $P_i$ has an order of $p_i^{t_i}$, where $p_i$ is a prime number. The generalized quaternion group $Q_{2^k}$ is defined as $Q_{2^k} = \langle x, y\rangle,$ where
\begin{enumerate}
	\item
	$x$ has order $2^{k-1}$ and $y$ has order $4,$
	\item
	every element of $Q_{2^k}$ can be written in the form $x^a$ or $x^ay$ for some $a\in \mathbb{Z},$
	\item
	$x^{2^{k-2}}=y^2,$
	\item
	for each $g\in Q_{2^k}$ with $g\in \langle x \rangle,$ we have $gxg^{-1}=x^{-1}.$
\end{enumerate}
For more information about $Q_{2^k}$ see \cite{generalized-quaternion, algebradummitfoote, scott-group}. 
Let $G_1$ be a finite nilpotent group that does not contain any Sylow subgroups that are either cyclic or generalized quaternion. For each Sylow subgroup $P_i$ in $G_1$, let $r_i$ denote the number of distinct $p_i$-ordered subgroups. It is assumed that $r_1\leq r_2\leq \cdots\leq r_m.$ Now, for a finite nontrivial nilpotent group $G,$ we have the following cases:
\begin{enumerate}
	\item 
	No Sylow subgroups of $G$ are either cyclic or generalized quaternion. In this scenario, $G=G_1.$
	\item
	$G$ has cyclic Sylow subgroups. In this case, $G = G_1\times \mathbb{Z}_n,$ where $G_1$ is described as above and $\text{gcd}(|G_1|, n)=1.$
	\item
	$G$ has a Sylow subgroup isomorphic to generalized quaternion. Here, $G = G_1\times Q_{2^k},$ and $\text{gcd}(|G_1|, 2)=1.$
	\item
	$G$ has both a cyclic Sylow subgroup and a Sylow subgroup isomorphic to generalized quaternion. In this situation, $G = G_1\times \mathbb{Z}_n\times Q_{2^k},$ where $\text{gcd}(|G_1|, n)=\text{gcd}(|G_1|, 2)=\text{gcd}(n, 2)=1.$
\end{enumerate}
Now, we will proceed to introduce three theorems that fully establish the strong domination number of proper enhanced power graphs of finite nilpotent groups. Initially, our focus will be on the nilpotent groups $G$ that do not possess any Sylow subgroups that are either cyclic or generalized quaternion.
\begin{theorem}\label{Thm: strong dom enhan of G1 and p grp nei cylic nor Q}
	Let $G$ be a finite nilpotent group having no Sylow subgroups which are either cyclic or generalized quaternion. Then 
	\begin{equation*}
		\gamma^{\text{strong}}(\mathcal{G}_E^{**}(G))= 
		\begin{cases}
			2r, & \text{ if } G \text{ is a } p\text{-group having } r \text{ distinct } p\text{ order subgroups }  \\
			\left(\displaystyle \min _{ 1 \leq i \leq m}  r_i\right)+1, & \text{ if } G=G_1=P_1\times \cdots\times P_m, m\geq 2.
		\end{cases}
	\end{equation*} 	
\end{theorem}
Now let us consider the groups $G$ that have only cyclic Sylow subgroups. In this particular case, we can state the following theorem:
\begin{theorem}\label{Thm: strong dom of enh pwr cylic sylow only}
	Let $G$ be a finite nilpotent group such that either $G=\mathbb{Z}_n$ or $G=G_1\times \mathbb{Z}_n,$ where $G_1=P_1\times \cdots \times P_m (m\geq 1)$ and $\text{gcd}(|G_1|, n)=1.$ Then 
	\begin{equation*}
		\gamma^{\text{strong}}(\mathcal{G}_E^{**}(G))= 
		\begin{cases}
			0, & \text{ if } G=\mathbb{Z}_n\\
			\left(\displaystyle \min _{ 1 \leq i \leq m}  r_i\right)+1, & \text{ if } G=G_1\times \mathbb{Z}_n\\
		\end{cases}
	\end{equation*} 	
\end{theorem}
In the current investigation, we explore the nilpotent groups $G$ that have Sylow subgroups of the generalized quaternion nature. In these particular cases, we compute both the strong domination number and the domination number. 
\begin{theorem}\label{Thm: strong dom enhanced pwr nilpotent grp having Q2k}
	Let $G$ be a finite nilpotent group having a Sylow subgroup which is a generalized quaternion. Then 
	\begin{equation*}
		\gamma^{\text{strong}}(\mathcal{G}_E^{**}(G))= 
		\begin{cases}
			2^{k-1}+2, & \text{ if either } G=Q_{2^k} \text{ or }\\& G=\mathbb{Z}_n\times Q_{2^k} \text{ and } \text{gcd}(n, 2)=1 \\
			\text{ min }\{r_1, 2^{k-2}+1\}+1, & \text{ if either } G=G_1 \times Q_{2^k}\text{ and } \text{gcd}(|G_1|, 2)=1\\& \text{ or } G=G_1\times \mathbb{Z}_n\times Q_{2^k} \text{ and } \\&\text{gcd}(|G_1|, n)=\text{gcd}(|G_1|, 2)=\text{gcd}(n, 2)=1.
		\end{cases}
	\end{equation*} 	
\end{theorem}
\begin{theorem}\label{Thm:dom enhanced pwr nilpotent grp having Q2k}
	Let $G$ be a finite nilpotent group having a Sylow subgroup which is a generalized quaternion. Then 
	\begin{equation*}
		\gamma(\mathcal{G}_E^{**}(G))= 
		\begin{cases}
			2^{k-2}+1, & \text{ if either } G=Q_{2^k} \text{ or }\\& G=\mathbb{Z}_n\times Q_{2^k} \text{ and } \text{gcd}(n, 2)=1 \\
			\text{ min }\{r_1, 2^{k-2}+1\}, & \text{ if either } G=G_1 \times Q_{2^k}\text{ and } \text{gcd}(|G_1|, 2)=1\\& \text{ or } G=G_1\times \mathbb{Z}_n\times Q_{2^k} \text{ and } \\&\text{gcd}(|G_1|, n)=\text{gcd}(|G_1|, 2)=\text{gcd}(n, 2)=1.
		\end{cases}
	\end{equation*} 	
\end{theorem}
\section{Preliminaries}\label{Sec:Preliminaries} 
In this section, we begin by revisiting certain previously established findings that are relevant to our paper. Bera et al. conducted a comprehensive study on various intriguing properties of enhanced power graphs of finite groups. Their works \cite{enhancedpwrgrapbb3} and \cite{bera-dey-sajal} effectively showcased the following results:
\begin{lemma}[Theorem 2.4, \cite{enhancedpwrgrapbb3}]\label{lemma:enhd coplte iff cyclic}
	The enhanced power graph $\mathcal{G}_E(G)$ of the group $G$ is complete if and only if $G$ is cyclic.
\end{lemma}
\begin{lemma}[Lemma 2.5,  \cite{bera-dey-sajal}]
	\label{lemma:any_odtwo_commuting_elements_of_prime_order_adjacent}
	Let $G$ be a finite group and $x, y \in G \setminus
	\{e\}$ be such that $\text{gcd}(\text{o}(x), \text{o}(y)) = 1 $ and $ xy = yx.$ Then, $x \sim y$ in $\mathcal{G}_E^*(G).$ 
\end{lemma}
\begin{lemma}[Lemma 2.6,  \cite{bera-dey-sajal}]
	\label{lema: p and p^i orderd path connd abelong< b>}
	Let $G$ be a $p$-group. Let $a, b$ be two elements of $G$ of order $p, p^i (i\geq 1)$ respectively. If there is a path between $a$ and $b$ in $\mathcal{G}_E^*(G),$ then $\langle a\rangle \subseteq \langle b\rangle.$ In particular, if both a and b have order p, then, $\langle a \rangle =\langle b \rangle.$
\end{lemma}
\begin{lemma}[Theorem 3.3, \cite{enhancedpwrgrapbb3}]\label{dom of gen Q_2^n in enhced}
	Let $G$ be a non-abelian $2$-group. Then the enhanced power graph $\mathcal{G}_E(G)$ is dominatable if and only if $G$ is a generalized quaternion group. In this case, identity and the unique element of order $2$ are dominating vertices.
\end{lemma}
\begin{lemma}[Theorem 4.2, \cite{bera-dey-sajal}]\label{lemm:conn and compo generalized qtrn}
	For $n\geq 3,$ let $Q_{2^k}$ be the generalized quaternion group. Then the vertex connectivity of $\mathcal{G}_E(Q_{2^k})$ is $2.$ Moreover the number of components of $\mathcal{G}^{**}_E(Q_{2^k})$ is $2^{k-2}+1.$ 	
\end{lemma}
In \cite{Costanzo-BullAust}, Costanzo et. al pointed out the following necessary and sufficient conditions for characterizing a group with a dominatable enhanced power graph. 
\begin{theorem}[Theorem 1.4, \cite{Costanzo-BullAust}]\label{thm:Costanzo}
	Let $G$ be a group, $g \in G$ and $\pi= \pi(\text{o}(g)).$ Write $g = \prod _{ p \in \pi} g_p,$ where each $g_p$ is a $p$-element for $p \in \pi$ and $g_pg_q = g_qg_p$  for all $p, q \in \pi.$ Then $g$ is a dominating vertex for $\mathcal{G}_E^{*}(G)$ if and only if, for each $p \in \pi,$ a Sylow $p$-subgroup $P$ of $G$ is cyclic or generalized quaternion and $ \langle g_p \rangle \leq P \cap Z(G).$ 
\end{theorem}
Let $G_1$ be a nilpotent group such that no Sylow subgroups of $G$ are either cyclic or generalized quaternion. Suppose $e, e''$ are the identities of the groups $G_1, Q_{2^k}$ respectively and $y$ is the unique $2$-ordered element of $Q_{2^k}.$ Let $D_1=\{(e, x, e''): x\in \mathbb{Z}_n\}$ and $D_2=\{(e, x, y): x \in \mathbb{Z}_n\}.$ As a corollary of Theorem \ref{thm:Costanzo} and Lemma \ref{dom of gen Q_2^n in enhced}, the authors \cite{bera-dey-JGT} had the  following characterization of dominating vertices of the enhanced power graph of any nilpotent group. 
\begin{corollary}[Corollary 4.2, \cite{bera-dey-JGT}]\label{Cor: Dom set for nilpotent grp}
	Let $G$ be a finite nilpotent group. Then 
	\begin{equation*}
		\text{Dom}(\mathcal{G}_E(G))= 
		\begin{cases}
			\{e\}, & \text{ if } G=G_1 \\
			\{(e, x): x \in \mathbb{Z}_n\},& \text{ if } G=G_1\times \mathbb{Z}_n\text{ and } \text{gcd}(|G_1|, n)=1 \\ 
			\{(e, e''), (e, y)\}, & \text{ if } G=G_1 \times Q_{2^k}\text{ and } \text{gcd}(|G_1|, 2)=1 \\
			D_2\cup D_3, & \text{ if } G=G_1\times \mathbb{Z}_n\times Q_{2^k} \text{ and } \\&\text{gcd}(|G_1|, n)=\text{gcd}(|G_1|, 2)=\text{gcd}(n, 2)=1.
		\end{cases}
	\end{equation*}
\end{corollary}
We will now demonstrate several crucial findings that serve as the basis for establishing our primary theorems.
\begin{lemma}\label{lem:no of compo in deleted enhanced p grp}
Let $G$ be a finite $p$-group such that $G$ is neither cyclic nor generalized quaternion group. Then number of components in the graph $\mathcal{G}_E^{**}(G)$ is $r,$ where $r$ is the number of distinct cyclic subgroups of order $p.$
\end{lemma}
\begin{proof}
Let $H_1=\langle a_1\rangle, \cdots, H_r=\langle a_r \rangle$ represent the distinct cyclic subgroups of order $p.$ For every $i\in [r],$ we define $K_i=\{x\in V(\mathcal{G}_E^{**}(G)): a_i\in \langle x\rangle\}.$ It is evident that each $K_i$ is connected (since any $x\in K_i$ is adjacent to $a_i).$ Furthermore, it is clear that $K_i\cap K_j=\emptyset,$ for $i, j\in [r], i\neq j.$ Suppose $a\in K_i$ and $b\in K_j.$ If there exists a path between $a$ and $b,$ then there exist two vertices $a'\in K_i$ and $b'\in K_j$ such that $a'\sim b'$ in $\mathcal{G}_E^{**}(G).$ Consequently, there exists $c\in  V(\mathcal{G}_E^{**}(G))$ such that $a', b'\in \langle c\rangle.$ In other words, $a'=c^{k_1}$ and $b'=c^{k_2}.$ Additionally, $a'\in K_i$ and $b'\in K_j$ imply that $a_i=c^{t_1}$ and $a_j=c^{t_2}.$ Hence, $a_i, a_j\in \langle c\rangle.$ This leads to a contradiction since $\langle a_i \rangle \neq \langle a_j \rangle$ and $\text{o}(a_i)=p=\text{o}(a_j).$ Therefore, each $K_i$ is a component. This concludes the proof.
\end{proof}	
\begin{lemma}\label{lem: no of compo in G1(p grp) times cylic enh pwr graph}
Let $G=G_1\times \mathbb{Z}_n$ be a nilpotent group such that $G_1$ is a $p$-group and $\text{gcd}(|G_1|, n)=1.$ Then number of components in $\mathcal{G}_E^{**}(G)$ is $r,$ where $r$ is the number of distinct $p$-ordered subgroups of $G_1.$
\end{lemma}
\begin{proof}
Consider $H_1=\langle a_1\rangle, \cdots, H_r=\langle a_r \rangle$ as the different cyclic subgroups of order $p$ of $G_1.$ For each $i\in [r],$ let's define $K_i=\{x\in  V(\mathcal{G}_E^{**}(G_1)): a_i\in \langle x\rangle\}.$ Our assertion is that $K_1\times \mathbb{Z}_n, \cdots, K_{r}\times \mathbb{Z}_n$ are distinct components of the graph $\mathcal{G}_{E}^{**}(G).$ Suppose $(a, b)\in K_i\times \mathbb{Z}_n.$ We will prove that $(a, b)\sim (a_i, e'),$ where $e'$ is the identity of the group $\mathbb{Z}_n.$ Let $\text{o}(a)=m.$ Now, if $\text{gcd}(m, n)=1,$ then $n^{\phi(m)}=mk+1,$ for some $k\in \mathbb{N}$ (according to Euler's formula). Therefore,
\begin{align*}
(a, b)^{n^{\phi(m)}}&=(a^{mk+1}, e')\\
		&=(a, e').
\end{align*}
Thus, $(a, e')\in \langle (a, b)\rangle.$ Moreover, since $a, a_i\in K_i$ and $\text{o}(a_i)=p,$ we have $a_i\in \langle a\rangle.$ Consequently, $(a_i, e')\in \langle (a, e')\rangle.$ As a result, $(a_i, e')\in \langle (a, b)\rangle.$ Therefore, $(a, b)\sim (a_i, e').$ Hence, $K_i\times \mathbb{Z}_n$ is connected for each $i\in [r].$ It is clear that $(K_i\times \mathbb{Z}_n)\cap (K_j\times \mathbb{Z}_n)=\emptyset,$ for $i\neq j.$ We will now demonstrate that there is no edge between any two vertices in $K_i\times \mathbb{Z}_n$ and $K_j\times \mathbb{Z}_n.$ Consider $(a, b)\in K_i\times \mathbb{Z}_n$ and $(a', b')\in K_i\times \mathbb{Z}_n,$ where $(a, b)\sim (a', b').$ This implies that there exists $(c, d)\in V(\mathcal{G}_E^{**}(G))$ such that $(a, b), (a', b')\in \langle (c, d)\rangle.$ Therefore, we have $a=c^{k_1}$ and $a'=c^{k_2}.$ Additionally, since $a\in K_i$ and $a'\in K_j,$ it follows that $a_i\in \langle a\rangle$ and $a_j\in \langle a'\rangle.$ Consequently,
\begin{align*}
	a_i&=a^{r_1}=c^{k_1r_1} \text{ and }\\
	a_j&=(a')^{r_2}=c^{k_2r_2}.
\end{align*}
This indicates that $a_i$ and $a_j$ belong to the cyclic subgroup $\langle c\rangle$, which contradicts our initial assumption that $\langle a_i\rangle$ and $\langle a_j\rangle$ are two distinct cyclic subgroups of order $p.$ 
\end{proof}
\section{Proofs of main theorems}\label{Se: proofs of main results}
\begin{proof}[Proof of Theorem \ref{Thm: existence of srong dom set in enhd pwr grph}]
Consider a $2$-group $G$ with an element of order $2$ that is not contained in any other proper subgroup of order $2^k$, where $k\geq 2$. Let $H_1=\langle a_1\rangle, \cdots, H_r=\langle a_r \rangle, (r\geq 3, \text{ see Lemma 3.10, \cite{bera-dey-JGT}})$ be the distinct cyclic subgroups of order $2$. For each $H_i$, we define the set $K_i$ as follows:
\[K_i=\{x\in V(\mathcal{G}_E^{**}(G)): a_i\in \langle x\rangle \}, \text{ for each } i \in [r].\] It is important to note that for $i\neq j$, $K_i\cap K_j=\emptyset$ and $V(\mathcal{G}_E^{**}(G))=\cup_{i=1}^r K_i.$ We assume that $a_1$ is not contained in any other proper subgroup of order $2^k, k\geq 2.$ Therefore, $K_1=\{a_1\}.$ Let $D$ be a total dominating set of the graph $\mathcal{G}_E^{**}(G)$. Then there exists an element $b\in D$ such that $b\sim a_1$, meaning that $a_1\in \langle b\rangle.$ However, it is given that $a_1$ is not contained in any other proper subgroup of order $2^k, k\geq 2.$ Hence, $a_1\in D.$ Now, since $K_i\cap K_j=\emptyset$, we can conclude that $|D|\geq 2.$ Let $c\in D$ such that $c\neq a_1.$ In this case, it can be shown that $a_1$ is not adjacent to $c.$ This contradicts the assumption that $D$ is a total dominating set.
	
To address the converse aspect, it is evident that if $G$ fails to meet any of the specified conditions, then it is certain that $|K_i|\geq 2$ for every $i\in [r]$. In such a scenario, we introduce $D'=\{a_{1}, a_1', a_2, a_2', \cdots, a_r, a_r'\}$, where $a_i, a_i'\in K_i$. It is clear that $D'$ serves as a total dominating set. Thus, the proof is concluded.	 	
\end{proof}
Now, based on the organization of a nilpotent group, we shall establish the subsequent propositions before proceeding to demonstrate the remaining theorems. 
\begin{proposition}\label{Prop: strong dom geq r1+1,G1 and g1 times cyclic, enhanced pwr}
Let $G$ be a finite nilpotent group that does not have a Sylow subgroup which is generalized quaternion. Then 
	\begin{equation*}
		\gamma^{\text{strong}}(\mathcal{G}_E^{**}(G))\geq
		\begin{cases}
			0, & \text{ if } G=\mathbb{Z}_n\\
			\left(\displaystyle \min _{ 1 \leq i \leq m}  r_i\right)+1, & \text{ if either } G=G_1, \text{ where } G_1=P_1\times \cdots \times P_m, m\geq 1\\& \text{ or } G=G_1\times \mathbb{Z}_n, \text{  and } \text{gcd}(|G_1|, n)=1.\\
		\end{cases}
	\end{equation*} 	
\end{proposition}
\begin{proof}
	First, let's consider the case where $G=\mathbb{Z}_n.$ In this case, $\mathcal{G}^{**}_E(G)$ is an empty graph, as stated in Lemma \ref{lemma:enhd coplte iff cyclic}. Therefore, it follows vacuously that $\gamma^{\text{strong}}(\mathcal{G}_E^{**}(\mathbb{Z}_n))\geq 0.$
	
	Now, let $G=G_1=P_1$ is a $p_1$-group and $r_1$ is the number of distinct $p_1$-ordered subgroups of $P_1.$ By Lemma \ref{lem:no of compo in deleted enhanced p grp} each $K_i$ is a component. Therefore, for each $a_i\in K_i,$ we have to choose another vertex say $a_i'$ from $K_i$ and clearly, $a_i\sim a_i'.$ So, $\gamma^{\text{strong}}(\mathcal{G}_{E}^{**} (G))\geq 2r_1.$ Therefore, $\gamma^{\text{strong}}(\mathcal{G}_{E}^{**} (G))\geq r_1+1.$
	
	In this portion, we consider that $G=G_1=P_1\times \cdots \times P_m,$ where $m\geq 2$ and no $P_i$ is either cyclic or generalized quaternion. For each $i\in [m],$ let $r_i$ be the number of distinct $p_i$-ordered cyclic subgroups of $G,$ and assume that $r_1\leq r_2\leq \cdots\leq r_m.$ Suppose \[D=\{(x_{11}, x_{12}, \cdots, x_{1m}), (x_{21}, x_{22}, \cdots, x_{2m}), \cdots, (x_{t1}, x_{t2}, \cdots, x_{tm})\}\] be a total dominating set of the graph $\mathcal{G}_{E}^{**} (G_1).$ Let $H_1=\langle a_1\rangle, \cdots, H_{r_1}=\langle a_{r_1} \rangle$ be the distinct cyclic subgroups of order $p_1.$ Now, for each $i\in [r_1], (a_i, e_2, \cdots, e_m)\in V(\mathcal{G}_{E}^{**} (G_1)),$ where $a_i$ is a generator of $H_i.$ Since, $D$ is a total dominating set then there exists a vertex say $(x_{i1}, x_{i2}, \cdots, x_{im})\in D$ such that $(x_{i1}, x_{i2}, \cdots, x_{im})\sim (a_i, e_2, \cdots, e_m),$ where $e_i$ is the identity of the subgroup $P_i.$ So, there exists $(y_{i1}, y_{i2}, \cdots, y_{im})\in V(\mathcal{G}_{E}^{**} (G_1))$ such that \[(x_{i1}, x_{i2}, \cdots, x_{im}), (a_i, e_2, \cdots, e_m)\in \langle (y_{i1}, y_{i2}, \cdots, y_{im})\rangle.\] That is; $x_{i1}, a_i\in \langle y_{i1}\rangle.$ Therefore, for each $i\in [r_1]$  $a_i\in \langle x_{i1}\rangle$ as $\text{o}(a_i)=p_1.$ Again, there are $r_1$ distinct $p_1$-ordered subgroups, and for distinct $i, j\in [r_1], \text{o}(a_i)=p_1=\text{o}(a_j)$ and $\langle a_i\rangle\neq \langle a_j\rangle.$ Consequently, $t\geq r_1$ (moreover, it is clear that $x_{i1}\neq x_{j1}$ for distinct $i, j\in [r_1]).$ Now, we show that $t>r_1.$ Let $t=r_1.$ Since, $D$ is total dominating set $\exists$ $(x_{j1}, x_{j2}, \cdots, x_{jm})\in D$ such that 
	$(x_{11}, x_{12}, \cdots, x_{1m})\sim (x_{j1}, x_{j2}, \cdots, x_{jm}), (\text{ clearly, } x_{11}\neq x_{j1}.)$ Therefore, \[(x_{11}, x_{12}, \cdots, x_{1m}), (x_{j1}, x_{j2}, \cdots, x_{jm})\in \langle(y_{s1}, y_{s2}, \cdots, y_{sm})\rangle,\]
	for some $(y_{s1}, y_{s2}, \cdots, y_{sm})\in V(\mathcal{G}_{E}^{**} (G_1)).$ So, $x_{11}, x_{j1}\in \langle y_{s1}\rangle.$ Now, $a_1\in \langle x_{11}\rangle$ and $a_j\in \langle x_{j1}\rangle$ both imply that $a_1, a_j\in \langle y_{s1}\rangle,$ a contradiction (as $\text{o}(a_i)=p_1=\text{o}(a_j)$ and $\langle a_i\rangle\neq \langle a_j\rangle$). Therefore, there is no vertex $(x_{j1}, x_{j2}, \cdots, x_{jk})$ in $D$ such that \[(x_{11}, x_{12}, \cdots, x_{1m})\sim (x_{j1}, x_{j2}, \cdots, x_{jm}).\] Consequently, $t\geq r_1+1.$ That is; 
	\begin{equation}\label{eqn: strong dom of G1 geq r1+1}
		\gamma^{\text{strong}}(\mathcal{G}_E^{**}(G_1))\geq \left(\displaystyle \min _{ 1 \leq i \leq m}  r_i\right)+1.
	\end{equation}
	Upon considering the scenario where $G=G_1\times \mathbb{Z}_n,$ we can demonstrate, in a similar manner as mentioned earlier, that \[\gamma^{\text{strong}}(\mathcal{G}_E^{**}(G_1))\geq \left(\displaystyle \min _{ 1 \leq i \leq m} r_i\right)+1.\] Thus, the proof is concluded.
\end{proof}
\begin{proposition}\label{Prop: strong dom enhanced pwr nilpotent grp having Q2k}
	Let $G$ be a finite nilpotent group having a Sylow subgroup which is generalized quaternion. Then 
	\begin{equation*}
		\gamma^{\text{strong}}(\mathcal{G}_E^{**}(G))\geq
		\begin{cases}
			\text{ min }\{r_1, 2^{k-2}+1\}, & \text{ if either } G=G_1\times \mathbb{Z}_n\times Q_{2^k} \text{ and } \\&\text{gcd}(|G_1|, n)=\text{gcd}(|G_1|, 2)=\text{gcd}(n, 2)=1\\& \text{ or } G=G_1 \times Q_{2^k}\text{ and } \text{gcd}(|G_1|, 2)=1\\
			2^{k-2}+1, & \text{ if either } G=Q_{2^k} \text{ or }\\& 
			G=\mathbb{Z}_n\times Q_{2^k} \text{ and } \text{gcd}(n, 2)=1 \\
		\end{cases}
	\end{equation*} 	
\end{proposition}
\begin{proof}
First suppose that $G=G_1\times \mathbb{Z}_n\times Q_{2^k}.$ Let \[D=\{(x_{11},\cdots, x_{1m},g_1, y_1), \cdots, (x_{t1},\cdots, x_{tm}, g_t,  y_t)\}\] be a total dominating set of the graph $\mathcal{G}_E^{**}(G).$ Suppose $\text{min}\{r_1, 2^{k-2}+1\}=r_1.$ Let $H_1=\langle a_{11}\rangle, \cdots, H_{r_1}=\langle a_{r_11}\rangle$ be the number of distinct cyclic subgroups of order $p_1$ in $P_1.$ Since $D$ is a dominating set, then $\exists$ $(x_{i1},\cdots, x_{im}, g_i, y_i)\in D$ such that \[(a_i, e_2, \cdots, e_m, e', e'')\sim (x_{i1},\cdots, x_{im}, g_i, y_i).\] So, $\exists$ $(x_{s1},\cdots, x_{sm}, g_s, y_s)\in V(\mathcal{G}_E^{**}(G_1\times\mathbb{Z}_n\times Q_{2^k}))$ such that \[(a_{i1}, e_2, \cdots, e_m, e', e''), (x_{i1},\cdots, x_{im}, g_i, y_i)\in \langle (x_{s1},\cdots, x_{sm}, g_s, y_s)\rangle.\] It follows that $a_{i1}, x_{i1}\in \langle x_{s1}\rangle.$ Now, $P_1$ is a $p_1$-group and $a_{i1}, x_{i1}\in \langle x_{s1}\rangle,$ and $\text{o}(a_{i1})=p_1.$ Therefore, $a_{i1}\in \langle x_{i1}\rangle.$ As a result, $t\geq r_1,$ as the number of distinct $p_1$-ordered cyclic subgroups in $P_1$ is $r_1.$ Now, we consider the set \[D'=\{(x_{i1}, \cdots, x_{im}, g_i, y_i)\in D: a_{i1}\in \langle x_{i1}\rangle\}, \text{ for } i\in [r_1].\] Our assertion is that $D'$ forms an independent set. Let's assume that \[(x_{i1}, \cdots, x_{im}, y_i), (x_{j1}, \cdots, x_{jm}, y_j)\in D',\] such that $(x_{i1}, \cdots, x_{im}, g_i, y_i)\sim (x_{j1}, \cdots, x_{jm}, g_j, y_j).$ So, \[\exists (x_{q1}, \cdots, x_{qm}, g_q, y_q)\in V(\mathcal{G}_E^{**}(G_1\times \mathbb{Z}_n\times Q_{2^k}))\] such that $x_{i1}, x_{j1}\in \langle x_{q1}\rangle.$ It follows that either $x_{i1}\in \langle x_{j1}\rangle$ or $x_{j1}\in \langle x_{i1}\rangle.$ If $x_{i1}\in \langle x_{j1}\rangle,$ then $a_{i1}\in \langle x_{j1}\rangle,$ a contradiction as $a_{j1}\in \langle x_{j1}\rangle$ and $\text{o}(a_{i1})=p_1=\text{o}(a_{j1})$ but $\langle a_{i1}\rangle\neq \langle a_{j1}\rangle.$ Likewise, if $x_{j1}\in \langle x_{i1}\rangle,$ then we can also obtain a contradiction. Consequently, $t\geq r_1.$ In the event that $\text{min}\{r_1, 2^{k-2}+1\}=2^{k-2}+1,$ we use $T_1=\langle b_1\rangle, \cdots, T_{2^{k-2}+1}=\langle b_{2^{k-2}+1}\rangle$ for $H_1=\langle a_{11}\rangle, \cdots, H_{r_1}=\langle a_{r_11}\rangle$ and similarly establish that $t\geq 2^{k-2}+1,$ where for each $i\in [2^{k-2}+1], \text{o}(b_i)=4$ and $\langle b_i\rangle\neq \langle b_j\rangle.$ Thus, $t\geq \text{min}\{r_1, 2^{k-2}+1\}+1.$

In the same manner, we can verify the case when $G=G_1\times Q_{2^k}.$

Now, we consider the case $G=\mathbb{Z}_n\times Q_{2^k},$ where $\text{gcd}(n, 2)=1.$ Let \[D=\{(x_1, y_1), \cdots, (x_t, y_t)\}\] be a total dominating set of $\mathcal{G}_E^{**}(\mathbb{Z}_n\times Q_{2^k}).$ Let $\text{o}(y_i)=2^m,$ then $m\geq 2,$ for all $i\in [t],$ otherwise by Corollary \ref{Cor: Dom set for nilpotent grp}, $(x_i, y_i)$ would be a dominating vertex of $\mathcal{G}_E(\mathbb{Z}_n\times Q_{2^k}).$ Let $T_1=\langle b_1\rangle, \cdots, T_{2^{k-2}+1}=\langle b_{2^{k-2}+1}\rangle$ be the distinct $4$-ordered subgroups of $Q_{2^k}.$ Since $D$ is a total dominating set, then for each $(e, b_i)\in V(\mathcal{G}_E^{**}(\mathbb{Z}_n\times Q_{2^k}))$ there exists $(x_i, y_i)\in D$ (say) such that $(x_i, y_i)\sim (e, b_i).$ So, there exists $(c_i, d_i)\in  V(\mathcal{G}_E^{**}(\mathbb{Z}_n\times Q_{2^k}))$ such that $(x_i, y_i), (e, b_i)\in \langle (c_i, d_i)\rangle.$ Therefore, $y_i, b_i\in \langle d_i\rangle.$ But $\text{o}(y_i)=2^m, m\geq 2.$  So, $b_i\in \langle y_i\rangle.$ Again, the number of distinct $4$-ordered subgroups in $Q_{2^k}$ is $2^{k-2}+1.$ As a result, $\gamma^{\text{strong}}(\mathcal{G}_E^{**}(\mathbb{Z}_n\times Q_{2^k}))\geq 2^{k-2}+1.$ This completes the proof. 	
\end{proof}
\begin{proof}[Proof of Theorem \ref{Thm: strong dom enhan of G1 and p grp nei cylic nor Q}]
Let $G$ be a $p$-group such that the graph $\mathcal{G}_{E}^{**}(G)$ has a total dominating set. Let $H_1=\langle a_1\rangle, \cdots, H_r=\langle a_r \rangle$ be the distinct cyclic subgroups of order $p.$ Let \[K_i=\{x\in G^{**}: a_i\in \langle x\rangle\}.\] Since $\mathcal{G}_{E}^{**}(G)$ has a total dominating set, then from Theorem \ref{Thm: existence of srong dom set in enhd pwr grph} it is clear that $|K_i|\geq 2,$ for each $i\in [r].$ Consider the set \[D=\{a_1, a_1', a_2, a_2', \cdots, a_r, a_r'\},\] where $H_i=\langle a_i\rangle$ and $a_i'\in K_i.$ Our claim is that $D$ is a total dominating set. Let $x\in G^{**}.$ Then $x\in K_i,$ for some $i\in [r].$ That is; $x\sim a_i.$ Again, for each $a_i\in D, a_i\sim a_i'.$ Therefore, $D$ is a total dominating set. So, $\gamma^{\text{strong}}(\mathcal{G}_{E}^{**} (G))\leq 2r.$ Also, by Proposition \ref{Prop: strong dom geq r1+1,G1 and g1 times cyclic, enhanced pwr} $\gamma^{\text{strong}}(\mathcal{G}_{E}^{**} (G))\geq 2r.$ Thus, $\gamma^{\text{strong}}(\mathcal{G}_{E}^{**} (G))=2r.$
	
\noindent
In this portion we consider that $G=G_1=P_1\times \cdots\times P_m, m\geq 2.$ By Proposition \ref{Prop: strong dom geq r1+1,G1 and g1 times cyclic, enhanced pwr}
	\begin{equation}\label{eqn:strong dom enh G1 geq r1+1}
		\gamma^{\text{strong}}(\mathcal{G}_{E}^{**} (G))\geq \left(\displaystyle \min _{ 1 \leq i \leq m}  r_i\right)+1.	
	\end{equation}	
	Moreover, using Lemma \ref{lemma:any_odtwo_commuting_elements_of_prime_order_adjacent} it can be shown that the set \[D'=\{(a_1, e_2, \cdots e_k), (a_2, e_2, \cdots, e_k), \cdots, (a_{r_1}, e_2, \cdots e_k), (e_1, x, e_3, \cdots, e_k)\}\] (where $x(\neq e_2)$ be an arbitrary element of the subgroup $P_2)$ is a total dominating set. Thus, 
	\begin{equation}\label{eqn: strong dom of G1 leq r1+1}
		\gamma^{\text{strong}}(\mathcal{G}_E^{**}(G_1))\leq \left(\displaystyle \min _{ 1 \leq i \leq m}  r_i\right)+1.	
	\end{equation}
	Hence, by Equations \eqref{eqn:strong dom enh G1 geq r1+1} and \eqref{eqn: strong dom of G1 leq r1+1} \[\gamma^{\text{strong}}(\mathcal{G}_E^{**}(G_1))=\left(\displaystyle \min _{ 1 \leq i \leq m}  r_i\right)+1.\]
\end{proof}

\begin{proof}[Proof of Theorem \ref{Thm: strong dom of enh pwr cylic sylow only}]
	If $G=\mathbb{Z}_n,$ then $\mathcal{G}_E^{**}(\mathbb{Z}_n)$ is an empty graph. So, vacuously \[\gamma^{\text{strong}}(\mathcal{G}_E^{**}(\mathbb{Z}_n))=0.\] Now, suppose that $G=G_1\times \mathbb{Z}_n,$ where $G_1=P_1\times\cdots \times P_m, m\geq 1$ and each $P_i$ is a $p_i$-group which is neither cyclic nor genralized quaternion, and $r_1\leq r_2\leq \cdots\leq r_m.$ Also, it is given that $\text{gcd}(|G_1|, n)=1.$ Let $H_1=\langle a_1\rangle, \cdots, H_{r_1}=\langle a_{r_1} \rangle,$ be the distinct cyclic subgroups of order $p_1.$ First suppose that $m=1.$ In this case, using Lemmas \ref{lemma:any_odtwo_commuting_elements_of_prime_order_adjacent} and \ref{lem: no of compo in G1(p grp) times cylic enh pwr graph}, one can prove that $\gamma^{\text{strong}}(\mathcal{G}_E^{**}(G_1\times \mathbb{Z}_n))=r_1+1.$ Now, suppose that $m\geq 2.$ We consider the set \[D=\{(a_1, e_2, \cdots, e_m, e'), \cdots, (a_{r_1}, e_2, \cdots, e_m, e'), (e_1, e_2, \cdots, e_m, g)\},\] where $e_i, e'$ are the identities of the groups $P_i, \mathbb{Z}_n$ respectively and $g$ is a generator of the group $\mathbb{Z}_n.$ Now, we prove that $D$ is a total dominating set. Let $(x_1, x_2, \cdots, x_m, y)$ be an arbitary element of $G^{**}\setminus D.$ Then we have the following cases:
	
	\noindent 
	\underline{Case 1: Let $(x_1, x_2, \cdots, x_m)\neq (e_1, e_2, \cdots, e_m)$ and $y=e'$:} 
	
	Suppose that $x_1\neq e_1.$ Here, we show that $(a_1, e_2, \cdots, e_m, e')\sim (x_1, x_2, \cdots, x_m, e').$ It is given that $|P_i|=p^{t_i}_i,$ for each $i\in [m].$ Since $\text{gcd}(p_1^{t_1}, p_2^{t_2}\cdots p_m^{t_m})=1,$ then $\exists$ $s_1, s_2\in \mathbb{Z}$ such that $s_1(p_1^{t_1})+s_2(p_2^{t_2}\cdots p_m^{t_m})=1.$ Now,
	\begin{align*}
		(x_1, x_2, \cdots, x_m, e')^{s_2(p_2^{t_2}\cdots p_m^{t_m})}&=(x_1^{s_2(p_2^{t_2}\cdots p_m^{t_m})}, e_2, \cdots, e_m, e')\\
		&=(x_1^{1-s_1(p_1^{t_1})}, e_2, \cdots, e_m, e')\\
		&=(x_1, e_2, \cdots, e_m, e')\\
		\Rightarrow	(x_1, e_2, \cdots, e_m, e')&\in \langle 	(x_1, x_2, \cdots, x_m, e')\rangle
	\end{align*}
	Again, $a_1\in \langle x_1\rangle$ (since, $x_1\neq e_1$ and $\text{o}(a_1)=p_1).$ Therefore, 
	\begin{align*}
		(a_1, e_2, \cdots, e_m, e')\in \langle (x_1, e_2, \cdots, e_m, e')\rangle
		&\Rightarrow  (a_1, e_2, \cdots, e_m, e')\in \langle (x_1, x_2, \cdots, x_m, e')\rangle\\
		&\Rightarrow (a_1, e_2, \cdots, e_m, e')\sim (x_1, x_2, \cdots, x_m, e').
	\end{align*}
	
	\noindent
	\underline{Case 2: Let $(x_1, x_2, \cdots, x_m)=(e_1, e_2, \cdots, e_m)$ and $y\neq e'$:} 
	
	In this case, it is eas to see that $(e_1, e_2, \cdots, e_m, y)\in \langle (e_1, e_2, \cdots, e_m, g)\rangle,$ (as $g$ is a generator of $\mathbb{Z}_n).$ Consequently, \[(e_1, e_2, \cdots, e_m, y)\sim (e_1, e_2, \cdots, e_m, g).\] 
	
	\noindent
	\underline{Case 3:  Let $(x_1, x_2, \cdots, x_m)\neq (e_1, e_2, \cdots, e_m)$ and $y\neq e'$:} 
	
	Here we prove that \[(x_1, x_2, \cdots, x_m, e'), (e_1, e_2, \cdots, e_m, y)\in \langle(x_1, x_2, \cdots, x_m, g)\rangle.  \text{ Clearly, }\] 	
	\begin{align*}\label{Eqn:strong dom case 3 G1 times cyclic}
		&(x_1, x_2, \cdots, x_m, g)^{p_1^{t_1}p_2^{t_2}\cdots p_m^{t_m}}=(e_1, e_2, \cdots, e_m, g^{p_1^{t_1}p_2^{t_2}\cdots p_m^{t_m}})\\
		&\Rightarrow (e_1, e_2, \cdots, e_m, g^{p_1^{t_1}p_2^{t_2}\cdots p_m^{t_m}})\in \langle (x_1, x_2, \cdots, x_m, g)\rangle.
	\end{align*}
	Now, $\text{gcd}(p_1^{t_1}p_2^{t_2}\cdots p_m^{t_m}, n)=1$ implies that $(e_1, e_2, \cdots, e_m, g^{p_1^{t_1}p_2^{t_2}\cdots p_m^{t_m}})$ is a generator of the cyclic group $\langle(e_1, e_2, \cdots, e_m, g) \rangle.$ So, $(e_1, e_2, \cdots, e_m, y)\in \langle (e_1, e_2, \cdots, e_m, g^{p_1^{t_1}p_2^{t_2}\cdots p_m^{t_m}})\rangle.$ Thus, \[(e_1, e_2, \cdots, e_m, y)\in \langle (x_1, x_2, \cdots, x_m, g)\rangle.\] 
	
	\noindent
	In this portion we show that $(x_1, x_2, \cdots, x_m, e')\in \langle(x_1, x_2, \cdots, x_m, g)\rangle.$ By the Euler's theorem, $n^{\phi(p_1^{t_1}p_2^{t_2}\cdots p_m^{t_m})}=\left(p_1^{t_1}p_2^{t_2}\cdots p_m^{t_m}\right)k+1, k\in \mathbb{N}$ (as $\text{gcd}(p_1^{t_1}p_2^{t_2}\cdots p_m^{t_m}, n)=1.$ Therefore,
	\begin{align*}
		(x_1, x_2, \cdots, x_m, g)^{n^{\phi(p_1^{t_1}p_2^{t_2}\cdots p_m^{t_m})}}&=((x_1, x_2, \cdots, x_m)^{n^{\phi(p_1^{t_1}p_2^{t_2}\cdots p_m^{t_m})}}, e')\\
		&=((x_1, x_2, \cdots, x_m)^{p_1^{t_1}p_2^{t_2}\cdots p_m^{t_m}k+1}, e')\\
		&=(x_1, x_2, \cdots, x_m, e').
	\end{align*}
	As a result, $(x_1, x_2, \cdots, x_m, e')\in \langle (x_1, x_2, \cdots, x_m, g) \rangle.$ Hence \[(x_1, x_2, \cdots, x_m, y)=(x_1, x_2, \cdots, x_m, e')(e_1, e_2, \cdots, e_m, y)\in \langle (x_1, x_2, \cdots, x_m, g)\rangle.\]
 It is necessary to demonstrate that if we select an element from the set $D$, there is another element in $D$ that is connected by an edge. To prove this, we can follow a similar approach as in Case 3. By doing so, we can establish that $(a_i, e_2, \cdots, e_m, e')\sim (e_1, e_2, \cdots, e_m, g)$ in $\mathcal{G}_E^{**}(G)$ for every $i\in [r].$ That is; $\gamma^{\text{strong}}(\mathcal{G}_{E}^{**} (G ))\leq r_1+1.$ Therefore,
	\begin{equation}\label{eqn:g cyclic times leq r1+1}
		\gamma^{\text{strong}}(\mathcal{G}_{E}^{**} (G ))\leq\left(\displaystyle \min _{ 1 \leq i \leq m}  r_i\right)+1.
	\end{equation}
	Again, by Proposition \ref{Prop: strong dom geq r1+1,G1 and g1 times cyclic, enhanced pwr} 
	\begin{equation}\label{eqn: g cyclic times geq r1+1}
		\gamma^{\text{strong}}(\mathcal{G}_{E}^{**} (G ))\geq\left(\displaystyle \min _{ 1 \leq i \leq m}  r_i\right)+1.
	\end{equation}
	Thus, by Equations \eqref{eqn:g cyclic times leq r1+1} and \eqref{eqn: g cyclic times geq r1+1} \[\gamma^{\text{strong}}(\mathcal{G}_{E}^{**}(G))=\left(\displaystyle \min _{ 1 \leq i \leq m}  r_i\right)+1.\]
\end{proof}	
In this portion, we will establish the validity of the case where the nilpotent group $G$ possesses a Sylow subgroup isomorphic a generalized quaternion. 

\begin{proof}[Proof of Theorem \ref{Thm: strong dom enhanced pwr nilpotent grp having Q2k}]
	First, let $G=Q_{2^k}.$ The group $Q_{2^k}$ has an element $x$ (say) such that $\text{o}(x)=2^{k-1}.$ Let $H=\langle x\rangle.$ Then $H$ is the only subgroup of $Q_{2^k}$ of order $2^{k-1}.$ Also, there are $2^{k-1}$ $4$-ordered elements in $Q_{2^{k}}\setminus H.$ Now, by Lemma \ref{lemm:conn and compo generalized qtrn}, the number of components in $\mathcal{G}_E^{**}(Q_{2^k})$ is $2^{k-2}+1$ which is equal to the number of distinct $4$-ordered subgroups of $Q_{2^k}.$ Therefore, $\gamma^{\text{strong}}(\mathcal{G}_E^{**}(G))\geq 2^{k-2}+1.$ Again it is clear that there is no edge between the vertices in $H$ and $Q_{2^n}\setminus H.$ Moreover, no two distinct $4$-ordered elements from two different $4$-ordered subgroups are edge connected. Therefore, to form a total dominating set we have to take at least two elements from each component in $\mathcal{G}_E^{**}(Q_{2^k}).$ It is possible as $\phi(4)=2.$ That is; $\gamma^{\text{strong}}(\mathcal{G}_E^{**}(G))\geq 2^{k-1}+2.$ 
	
	Now, we take \[D=\{a_1, a_1', a_2, a_2', \cdots, a_{2^{k-2}+1}, a'_{2^{k-2}+1}\},\] where for each $i\in [2^{k-2}+1], \text{o}(a_i)=4=\text{o}(a_i')$ and $\langle a_i\rangle=\langle a_i'\rangle$ but for $i\neq j, \langle a_i\rangle\neq \langle a_j\rangle.$ It is easy to see that $D$ is a total dominating set. Therefore, $\gamma^{\text{strong}}(\mathcal{G}_E^{**}(G))\leq 2^{k-1}+2.$ Hence, $\gamma^{\text{strong}}(\mathcal{G}_E^{**}(G))=2^{k-1}+2.$    
	
	Here we consider that $G=\mathbb{Z}_n\times Q_{2^k}$ and $\text{gcd}(n, 2)=1.$ Let $D=\{(x_1, y_1), \cdots, (x_t, y_t)\}$ be a total dominating set of $\mathcal{G}_E^{**}(\mathbb{Z}_n\times Q_{2^k}).$
	
	\noindent
	\underline{Claim 1: $t> 2^{k-2}+1:$}
	Let $T_1=\langle b_1\rangle, \cdots, T_{2^{k-2}+1}=\langle b_{2^{k-2}+1}\rangle$ be the distinct $4$-ordered cyclic subgroups of $Q_{2^k}.$ So for each $i\in [2^{k-2}+1], \exists$ $(x_i, y_i)\in D$ such that $(e', b_i)\sim (x_i, y_i).$ Therefore $b_i\in \langle y_i\rangle,$ since $\text{o}(y_i)=2^m, m\geq 2.$ So $t\geq 2^{k-2}+1.$
	Let \[D'=\{(x_i, y_i)\in D: b_i\in \langle y_i\rangle \text{ for each } i\in [2^{k-2}+1]\}.\] We will prove that $D'$ is an idependent subset of $D.$ Let $(x_i, y_i)\sim (x_j, y_j)$ for some $i, j\in [2^{k-2}+1].$ Then $\exists$ $(x, y)$ such that $(x_i, y_i), (x_j, y_j)\in \langle (x, y)\rangle.$ That is; $y_i, y_j\in \langle y\rangle.$ Let, $\text{o}(y_i)=2^{m_i}, \text{o}(y_j)=2^{m_j}.$ Then clearly $m_i, m_j\geq 2$ (otherwise, $(x_i, y_i), (x_j, y_j)\in\text{Dom}(\mathcal{G}_E(\mathbb{Z}_n\times Q_{2^k}))$ by Corollary \ref{Cor: Dom set for nilpotent grp}). Now, either $m_i\geq m_j$ or $m_j\geq m_i.$ Let $m_i\geq m_j.$ Then $y_i, y_j\in \langle y\rangle$ implies that $y_j\in \langle y_i\rangle.$ Similarly, $y_i\in \langle y_j\rangle$ if $m_j\geq m_i.$ Therefore, for each $(x_i, y_i)\in D$ we have to choose $(x_j, y_j)\in D$ such that either $y_j\in \langle y_i\rangle$ or $y_i\in \langle y_j\rangle.$ Again $b_j\in \langle y_j\rangle.$ So, $b_j\in \langle y_i\rangle,$ a contradiction, as $b_i\in \langle y_i\rangle, \text{o}(b_i)=\text{o}(b_j)=4$ but $\langle b_i\rangle\neq \langle b_j\rangle$ (similarly we get a contradiction if $y_i\in \langle y_j\rangle).$ That is, $D'$ is an independent subset of $D$ with $|D'|=2^{k-2}+1.$ Therefore, $t> 2^{k-2}+1.$
	
	\noindent
	\underline{Claim 2: $t\geq 2^{k-1}+2:$} Here we first prove that for any two distinct vertices $(x_i, y_i), (x_j, y_j)\in D',$ there is no vertex $(x_{\ell}, y_{\ell})\in D$ such that $(x_i, y_i)\sim (x_{\ell}, y_{\ell})\sim (x_j, y_j).$ Suppose that $(x_i, y_i)\sim (x_{\ell}, y_{\ell})\sim (x_j, y_j).$ Now, $(x_i, y_i)\sim (x_{\ell}, y_{\ell})$ implies that either $y_i\in \langle y_{\ell}\rangle$ or $y_{\ell}\in \langle y_{i}\rangle.$ Also, $(x_{\ell}, y_{\ell})\sim (x_j, y_j)$ implies that either $y_j\in \langle y_{\ell}\rangle$ or $y_{\ell}\in \langle y_{j}\rangle.$ So we have the following four cases.
	
	\noindent
	\underline{Case 1: Suppose $y_i\in \langle y_{\ell}\rangle$ and  $y_j\in \langle y_{\ell}\rangle$:} 

 \noindent
In this scenario, we can conclude that either $y_i\in \langle y_{j}\rangle$ or $y_j\in \langle y_{i}\rangle.$ If $y_i\in \langle y_{j}\rangle,$ then it follows that $b_i\in \langle y_{j}\rangle.$ However, this leads to a contradiction since $b_j\in \langle y_j\rangle, \text{o}(b_i)=\text{o}(b_j)=4,$ but $\langle b_i\rangle\neq \langle b_j\rangle.$ Similarly, if $y_j\in \langle y_i\rangle,$ we also arrive at a contradiction.
 
\noindent
	\underline{Case 2: Let's consider the situation where $y_i\in \langle y_{\ell}\rangle$ and $y_{\ell}\in \langle y_{j}\rangle$:}

 \noindent
 This implies that $y_i\in \langle y_{j}\rangle.$ However, this is not possible, as we have already established in the previous case.

	\noindent
	\underline{Case 3: Now, let's examine the case where $y_{\ell}\in \langle y_{i}\rangle$ and $y_j\in \langle y_{\ell}\rangle$:} 

 \noindent
 It is clear that $y_j\in \langle y_{i}\rangle.$ Similarly, we can deduce that this leads to a contradiction.
	
	\noindent
	\underline{Case 4: Finally, let's consider the scenario where $y_{\ell}\in \langle y_{i}\rangle$ and $y_{\ell}\in \langle y_{j}\rangle$:}

 \noindent
 Since, $\text{o}(y_{\ell})=2^m, m\geq 2,$ it can be deduced that there exists an element $b_s\in Q_{2^k}$ such that $\text{o}(b_s)=4$ and $b_s\in \langle y_{\ell}\rangle.$ As a result, $b_s\in \langle y_i\rangle$ and $b_s\in \langle y_j\rangle.$ Consequently, $b_i=b_s=b_j.$ This contradicts the fact that $b_i\neq b_j.$ Therefore, it can be concluded that $(x_i, y_i)$ and $(x_j, y_j)$ do not have a common neighbor in $D.$

Hence, for each $(x_i, y_i)\in D'$, it is necessary to select a unique $(x_i', y_i')$ such that $(x_i, y_i)\sim (x_i', y_i').$ This implies that
\begin{equation}\label{Eqn:strong dom enhanced cyclic times Q geq}
\gamma^{\text{strong}}(\mathcal{G}_E^{**}(\mathbb{Z}_n\times Q_{2^k}))\geq 2(2^{k-2}+1)=2^{k-1}+2.
\end{equation}

This concludes the proof of Claim 2.
	
	Now, we consider the following set: \[D''=\{(e, b_1), (e, b_1'), (e, b_2), (e, b_2'), \cdots, (e, b_{2^{k-2}+1}), (e, b'_{2^{k-2}+1})\},\] where for each $i\in [2^{k-2}+1],$ $T_i=\langle b_i\rangle=\langle b_i'\rangle$ is a $4$-ordered subgroup of $Q_{2^k}$ and $T_i\neq T_j.$  Clearly, for each $i\in [2^{k-2}+1]$ $(e, b_i)\sim (e, b_i').$ Let $(a, b)\in V(\mathcal{G}_E^{**}(\mathbb{Z}_n\times Q_{2^k}))\setminus D''.$ Then $\text{o}(b)=2^m, m\geq 2.$ So, $b_i\in \langle b\rangle,$ for some $i\in [2^{k-2}+1].$ Again $\text{gcd}(2^m, n)=1$ implies that $n^{\phi{(2^m)}}=2^mk+1$ (by Euler's Theorem). So,
	\begin{align*}
		(a, b)^{n^{\phi{(2^m)}}}=&(e, b^{n^{\phi{(2^m)}}})\\
		=&(e, b^{2^mk+1}\\
		=&(e, b).	
	\end{align*}    
	Therefore, $(e, b)\in \langle (a, b)\rangle.$ Also, $(e, b)^r=(e, b_i)$ (since, $b_i\in \langle b\rangle$ implies that $b_i=b^r,$ for some $r).$ Thus, $(e, b_i)\in \langle (a, b)\rangle.$ That is; $(a, b)\sim (e, b_i).$ Hence, $D''$ is a total dominating set. Consquently, 
	\begin{equation}\label{Eqn:strong dom enhanced cylic times Q leq}
		\gamma^{\text{strong}}(\mathcal{G}_E^{**}(\mathbb{Z}_n\times Q_{2^k}))\leq 2^{k-1}+2.
	\end{equation}
	So, by Equations \eqref{Eqn:strong dom enhanced cyclic times  Q geq} and \eqref{Eqn:strong dom enhanced cylic times Q leq} $\gamma^{\text{strong}}(\mathcal{G}_E^{**}(\mathbb{Z}_n\times Q_{2^k}))=2^{k-1}+2.$  
	
	In this part, we consider that $G=G_1\times Q_{2^{k}},$ where $G_1=P_1\times \cdots\times P_m.$ Using Proposition \ref{Prop: strong dom enhanced pwr nilpotent grp having Q2k} we can say that 
	\begin{equation}\label{Eqn: strong dom enh G1 times Q geq }
		\gamma^{\text{strong}}(\mathcal{G}_E^{**}(G_1\times Q_{2^k}))\geq \text{min}\{r_1, 2^{k-2}+1\}+1.
	\end{equation}
	In the second part, we consider the sets $D_1$ and $D_2$ as follows:
	\begin{align*}
		D_1=&\{(a_{11}, e_2, \cdots, e_m, e''),\cdots,  (a_{r_11}, e_2, \cdots, e_m, e''), (e_1, \cdots, e_m, b_i)\}\\
		D_2=&\{(e_1, \cdots, e_m, b_1),\cdots,  (e_1, \cdots, e_m, b_{2^{k-2}+1}), (a_{11}, \cdots, e_m, e'')\},
	\end{align*}
	where $\text{o}(b_i)=4, \langle b_i\rangle\neq \langle b_j\rangle,$ for all $i, j\in [2^{k-2}+1]$ and $\text{o}(a_{i1})=p_1, \langle a_{i1}\rangle\neq \langle a_{j1}\rangle,$ for all $i, j\in [r_1],$ and $e''$ is the identity of the group $Q_{2^k}.$ Now we take 
	\begin{equation*}
		D''=
		\begin{cases}
			D_1, & \text{ if  } \text{min}\{r_1, 2^{k-2}+1\}\\&=r_1 \\
			D_2, & \text{ if } \text{min}\{r_1, 2^{k-2}+1\}\\&=2^{k-2}+1.
		\end{cases}
	\end{equation*} 
	One can check that $D''$ is a total dominating set in both cases. Therefore, 
	\begin{equation}\label{Eqn: strong dom enh G1 times Q leq }
		\gamma^{\text{strong}}(\mathcal{G}_E^{**}(G_1\times Q_{2^k}))\geq \text{min}\{r_1, 2^{k-2}+1\}+1.
	\end{equation}
	Hence by Equations \ref{Eqn: strong dom enh G1 times Q geq } and \ref{Eqn: strong dom enh G1 times Q leq },  \[\gamma^{\text{strong}}(\mathcal{G}_E^{**}(G_1\times Q_{2^k}))=\text{min}\{r_1, 2^{k-2}+1\}+1.\]
	
	If $G=G_1\times \mathbb{Z}_n\times Q_{2^{k}},$ we can deduce from Proposition \ref{Prop: strong dom enhanced pwr nilpotent grp having Q2k}  that \[\gamma^{\text{strong}}(\mathcal{G}_E^{**}(G))\geq \text{min}\{r_1, 2^{k-2}+1\}+1.\] Moving on to the other part, let's consider the following sets:
	\begin{align*}
		D_1=&\{(a_{11}, e_2, \cdots, e_m, e', e''),\cdots,  (a_{r_11}, e_2, \cdots, e_m, e', e''), (e_1, \cdots, e_m, e', b_i)\}\\
		D_2=&\{(e_1, \cdots, e_m, e', b_1),\cdots,  (e_1, \cdots, e_m, e', b_{2^{k-2}+1}), (a_{11}, \cdots, e_m, e', e'')\},
	\end{align*}
	where $\text{o}(a_{i1})=p_1, \langle a_{i1}\rangle\neq \langle a_{j1}\rangle,$ for all $i, j\in [r_1],$ and $\text{o}(b_i)=4, \langle b_i\rangle\neq \langle b_j\rangle,$ for all $i, j\in [2^{k-2}+1]$ and $e'$ is the identity of the group $\mathbb{Z}_n.$ Let
	\begin{equation*}
		D''=	
		\begin{cases}
			D_1,&  \text{ if  } \text{min}\{r_1, 2^{k-2}+1\}\\&=r_1 \\
			D_2, & \text{ if } \text{min}\{r_1, 2^{k-2}+1\}\\&=2^{k-2}+1.
		\end{cases}
	\end{equation*} 
 It can be verified that $D''$ is a dominating set in both cases. Consequently, the theorem holds.
\end{proof}
\begin{proof}[Proof of Theorem \ref{Thm:dom enhanced pwr nilpotent grp having Q2k}]
	Proof of this theorem follows from the proof of Theorem \ref{Thm: strong dom enhanced pwr nilpotent grp having Q2k}.  
\end{proof}
\bibliographystyle{amsplain}
\bibliography{gen-inv-lcp.bib}

\providecommand{\bysame}{\leavevmode\hbox to3em{\hrulefill}\thinspace}
\providecommand{\MR}{\relax\ifhmode\unskip\space\fi MR }
\providecommand{\MRhref}[2]{%
  \href{http://www.ams.org/mathscinet-getitem?mr=#1}{#2}
}
\providecommand{\href}[2]{#2}
\begin{thebibliography}{10}

\bibitem{firstenhcedpwrstrctreaacns1}
G.~Aalipour, S.~Akbari, P.~J. Cameron, R.~Nikandish, and F.~Shaveisi, \emph{On
  the structure of the power graph and the enhanced power graph of a group},
  The Electronic J. Combinatorics \textbf{24} (2017), no.~3, P3.16.

\bibitem{surveypwrgraphkac1}
J.~Abawajy, A.~V. Kelarev, and M.~Chowdhury, \emph{Power graphs: a survey},
  Electron. J. Graph Theory Appl \textbf{1} (2013), 125--147.

\bibitem{D-Angel}
D.~Anger, \emph{Application of graph domination to defend medical information
  networks against cyber threats}, J. Ambient Intell. Humaniz Comput.

\bibitem{ijraeljofmathematics}
J.~Ara\'{u}jo, W.~Bentz, and J.~Konieczny, \emph{The commuting graph of the
  symmetric inverse semigroup}, Israel J. Mathematics \textbf{207} (2015),
  no.~1, 103--149.

\bibitem{finite-smple-group-book}
M.~Aschbacher, \emph{Finite group theory}, Cambridge University Press, 2000.

\bibitem{intersectionpwegraphb3}
S.~Bera, \emph{On the intersection power graph of a finite group}, Electron. J.
  Graph Theory Appl. \textbf{6} (2018), no.~1, 178--189.

\bibitem{enhancedpwrgrapbb3}
S.~Bera and A.~K. Bhuniya, \emph{On enhanced power graphs of finite groups}, J.
  Algebra Appl. \textbf{17} (2018), no.~8, 1850146, 8.

\bibitem{bera-dey-JGT}
S.~Bera and H.~K. Dey, \emph{On the proper enhanced power graphs of finite
  nilpotent groups}, J. Group Theory \textbf{25} (2022), no.~6, 1109--1131.

\bibitem{bera-dey-sajal}
S.~Bera, H.~K. Dey, and S.~K. Mukherjee, \emph{On the connectivity of enhanced
  power graphs of finite groups}, Graphs Combin. \textbf{37(2)} (2021),
  591--603.

\bibitem{graphthrybondymurti}
J.~A. Bondy and U.~S.~R. Murty, \emph{Graph theory}, Springer-Verlag, 2008.

\bibitem{braurflower}
R.~Brauer and K.A. Fowler, \emph{On groups of even order}, The Annals of
  Mathematics \textbf{62} (1955), no.~3, 567--583.

\bibitem{undpwrgraphofsemgmainsgc1}
I.~Chakrabarty, S.~Ghosh, and M.~K. Sen, \emph{Undirected power graphs of
  semigroups}, Semigroup Forum \textbf{78} (2009), 410--426.

\bibitem{generalized-quaternion}
K.~Conrad, \emph{Generalized quaternions},
  http://www.math.uconn.edu/~kconrad/blurbs/, 2014.

\bibitem{Costanzo-BullAust}
D.~G. Costanzo, M.~L. Lewis, S.~Schmidt, E.~Tsegaye, and G.~Udell, \emph{The
  cyclic graph of a z-group}, Bulletin of the Australian Mathematical Society
  \textbf{104(2)} (2021), 295--301.

\bibitem{das-saha-saba}
A.~Das, M.~Saha, and S.~Al-Kaseasbeh, \emph{On co-maximal subgroup graph of a
  group}, arXiv:2103.14284v2 (2021).

\bibitem{algebradummitfoote}
D.~S. Dummit and R.~M. Foote, \emph{Abstract algebra}, Wiley, 2003.

\bibitem{onboundingdiamcommuting}
M.~Giudici and A.~Pope, \emph{On bounding the diameter of the commuting graph
  of a group}, J. Group Theory \textbf{17} (2014), no.~1, 131--149.

\bibitem{algbraphgodsil}
C.~Godsil and G.~Royle, \emph{Algebraic graph theory}, Springer-Verlag Inc.,
  New York, 2001.

\bibitem{Hamzeh-ashrafi}
A.~Hamzeh and A.~R. Ashrafi, \emph{Automorphism groups of supergraphs of the
  power graph of a finite group}, European. J. Combin. \textbf{60} (2017),
  82--88.

\bibitem{cayleygraphsckry}
A.~Kelarev, J.~Ryan, and J.~Yearwood, \emph{Cayley graphs as classifiers for
  data mining: The influence of asymmetries}, Discrete Mathematics \textbf{309}
  (2009), no.~17, 5360--5369.

\bibitem{automatatheory}
A.~V. Kelarev, \emph{Graphs algebras and automata}, Marcel Dekker, New York,
  2003.

\bibitem{combinatorialpropertyandpowergraphsofgroupskq1}
A.~V. Kelarev and S.~J. Quin, \emph{A combinatorial property and power graphs
  of groups}, Contrib. General Algebra \textbf{12} (2000), 229--235.

\bibitem{directedgrphcompropofsemgrpkq3}
\bysame, \emph{Directed graph and combinatorial properties of semigroups}, J.
  Algebra \textbf{251} (2002), 16--26.

\bibitem{Jitendra-Ma-acta-enhaced}
J.~Kumar, X.~Ma, Parveen, and S.~Singh, \emph{Certain properties of the
  enhanced power graphs associated with a finite group}, Acta Math. Hungar
  \textbf{169} (2023), no.~1, 238--251.

\bibitem{ma-she}
X.~Ma and Y.~She, \emph{The metric dimension of the enhanced power graph of a
  finite group}, J. Algebra Appl. (2019), 190--198.

\bibitem{scott-group}
W.~R. Scott, \emph{Group theory}, Dover Publ., New York, 1987.

\bibitem{Segev-Seitz-Pacific}
Y.~Segev and G.~M. Seitz, \emph{Anisotropic groups of type an and the commuting
  graph of finite simple groups}, Pacific J. Math \textbf{14(1)} (2002),
  no.~202, 125--225.

\bibitem{graphthrywest}
D.~B. West, \emph{Introduction to graph theory}, Pearson Education Inc., 2001.

\bibitem{Zahirovienhnacedpwrgraph}
S.~Zahirovi\'{c}, I~Bo\v{s}njak, and R~Madar\'{a}sz, \emph{A study of enhanced
  power graphs of finite groups}, J. Algebra Appl. \textbf{19} (2020), no.~4,
  2050062, 20.

\end{thebibliography}
\end{document}